\documentclass[11pt,a4paper]{article}
\usepackage[top=2.5cm,bottom=2.5cm,left=2.2cm,right=2.2cm]{geometry}
\usepackage[T1]{fontenc}
\usepackage[utf8]{inputenc}
\usepackage{color}
\usepackage{graphicx}
\usepackage{amsfonts}
\usepackage{extarrows}
\usepackage{amsmath,amsthm,amssymb,color}
\usepackage{hyperref}
\usepackage{subcaption}
\usepackage{xcolor}
\usepackage{eepic}
\usepackage{lineno}
\usepackage{enumerate}	
\usepackage{paralist}
\usepackage{cite}
\usepackage{algorithm}
\usepackage{algorithmicx}
\usepackage{algpseudocode}
\usepackage{authblk}
\usepackage{mathtools}
\usepackage{pifont}
\usepackage[numbers,sort&compress]{natbib}

\usepackage{tikz}

\hypersetup
{
	colorlinks=true,
	linkcolor=blue,
	filecolor=blue,
	urlcolor=blue,
	citecolor=cyan,
}

\newtheorem{theorem}{Theorem}[section]

\newtheorem{lemma}[theorem]{Lemma}
\newtheorem{corollary}[theorem]{Corollary}
\newtheorem{conjecture}[theorem]{Conjecture}

\newtheorem{observation}[theorem]{Observation}

\theoremstyle{definition}

\newtheorem{claim}{\indent Claim}[theorem]      

\newtheorem{case}{\indent Case}[section]

{\setlength{\leftmargini}{1.5\parindent}
 \begin{itemize}
 \setlength{\itemsep}{-1.1mm}}
{\end{itemize}}

\begin{document}
	\title{\bf Bound vertices of longest paths between two vertices in cubic graphs}
	\author{\bf Chengli Li\footnote{Email: lichengli0130@126.com.}}
	\author{\bf Feng Liu\footnote{Email: liufeng0609@126.com (corresponding author).}}
	
	\affil{ Department of Mathematics,
		East China Normal University, Shanghai, 200241, China}
	\date{}
	\maketitle
\begin{abstract}
Thomassen's chord conjecture from 1976 states that every longest cycle in a $3$-connected graph has a chord. This is one of the most important unsolved problems in graph theory.   Let $H$ be a subgraph of a graph $G$. A vertex $v$ of $H$  is said to be $H$-bound if all the neighbors of $v$ in $G$ lie in $H$. Recently, Zhan has made the more general conjecture that in a $k$-connected graph, every longest path $P$ between two vertices contains at least $k-1$ internal $P$-bound vertices. In this paper, we prove that  Zhan's conjecture holds for $2$-connected cubic graphs. This conclusion generalizes a result of Thomassen [{\em J. Combin. Theory Ser. B} \textbf{129} (2018) 148--157]. Furthermore,  we prove that if the two vertices are adjacent, Zhan's conjecture holds for $3$-connected cubic graphs, from which we deduce that  every longest cycle in a $3$-connected cubic graph has at least two chords. This strengthens a  result of Thomassen [{\em J. Combin. Theory Ser. B} \textbf{71} (1997) 211--214].

			
\end{abstract}
	
\section{Introduction}

 All graphs considered in this paper are finite and simple. The order of a graph is its number of vertices, and the size is its number of edges. We denote by $V(G)$ and $E(G)$ the vertex set and edge set of a graph $G,$ respectively. For a subset $S$ of $V(G)$, we use $G[S]$ to denote the subgraph of $G$ induced by $S.$

 Let $S$ and $T$ be two subsets of $V(G)$, and let $x$  and $y$ be two vertices of $G$.  We use $N_S(x)$ to
denote the neighbors of $x$ in $S$, and define $N_S(T)=\cup_{x\in T}N_S(x)$. Denote by $[S,T]$ the set of edges having one
 endpoint in $S$ and the other in $T$.  Let $d_G(u)$  denote the degree of $u$ in $G$.
Furthermore, we denote the number of components in a graph $G$ by  $c(G)$. Let $G_1$ and $G_2$ be two vertex disjoint subgraphs of $G$. The  union $G_1\cup G_2$ denote the subgraph induced by $V(G_1)\cup V(G_2)$.  
Let $P$ be a path. We use  $P^*$ to denote  the set of internal vertices of $P$, and use $P[u,v]$ to denote the segment of $P$ between two  vertices $u$  and $v$. The symbol $[k]$ used in this article represents the set $\{1,2,\ldots,k\}$.

A Hamilton path in $G$ is a path containing every vertex of $G$.  A Hamilton  cycle in $G$ is a cycle containing every vertex of $G$. The graph $G$ is hamiltonian if it contains a Hamilton cycle. 

Thomassen's famous chord conjecture from 1976 is as follows.
\begin{conjecture}[Thomassen \cite{Alspach1985,Thomassen1989}]\label{Conjecture-Thomassen}
    Every longest cycle in a $3$-connected graph has a chord.
\end{conjecture}
Although this beautiful conjecture currently remains unsolved,  there are many partial results for which Conjecture~\ref{Conjecture-Thomassen} has been proved. See \cite{Thomassen1997,Thomassen2018,Birmele2008,Kawarabayashi2007,Li2003-1,Li2003-2,Wu2014,Zhang1987} for some previous and recent progress on this conjecture. In particular,  Conjecture~\ref{Conjecture-Thomassen} has been proved for cubic graphs by Thomassen.
\begin{theorem}[Thomassen  \cite{Thomassen1997}]\label{Theorem-Thomassen=3-connected}
Every longest cycle in a $3$-connected cubic graph has a chord.
\end{theorem}

Let $H$ be a subgraph of a graph $G$. A vertex $v$ of $H$  is said to be $H$-bound if all the neighbors of $v$ in $G$ lie in $H$; i.e., $N_G(v)\subseteq V(H)$. A vertex of a path $P$ is called an internal vertex if it is not an endpoint of $P$. For two distinct vertices $x$ and $y$, an $(x, y)$-path is a path whose endpoints are $x$ and $y$.  Recently, Zhan has proposed the following conjecture, which implies Conjecture~\ref{Conjecture-Thomassen}.

\begin{conjecture}[Zhan \cite{Zhan2024}]\label{Conjecture-Zhan}
Let $G$ be a $k$-connected graph with $k\geq 2$ and let $x,y$ be two distinct vertices of $G$. If $P$ is a longest $(x,y)$-path in $G$, then $P$ contains  $k-1$ internal $P$-bound vertices.
\end{conjecture}

In 2018, Thomassen \cite{Thomassen2018} proved that every longest cycle in a $2$-connected cubic graph has a chord. In this paper, we verify that Conjecture~\ref{Conjecture-Zhan} holds for $2$-connected cubic graphs. Our result shows that every longest path  between two vertices in a $2$-connected cubic graph has a chord, which generalizes Thomassen's result. We have used some ideas of Thomassen \cite{Thomassen2018} in our proofs, but we need to add new ones. Our first result is stated in the following theorem.

\begin{theorem}\label{Theorem-cubic-2-connected}
Let $G$ be a $2$-connected cubic graph and let $x,y$ be two distinct vertices of $G$. If  $P$ is a longest $(x,y)$-path in $G$, then $P$ contains at least one internal $P$-bound vertex.
\end{theorem}

Our second result shows that if the two vertices are adjacent, then Zhan's conjecture holds for $3$-connected cubic graphs. Specifically, we prove the following result.
 
\begin{theorem}\label{Theorem-cubic-3-connected}
Let $G$ be a $3$-connected cubic graph and let $x,y$ be two distinct vertices of $G$. If $x$ and $y$ are adjacent and $P$ is a longest $(x,y)$-path in $G$, then $P$ contains at least two internal $P$-bound vertices. 
\end{theorem}
As an application of Theorem~\ref{Theorem-cubic-3-connected}, we have the following corollary.
\begin{corollary}\label{Cororllary-cubic-3-connected-2-chord}
 Every longest cycle in a $3$-connected cubic graph has at least two chords.
\end{corollary}
\begin{proof}[\bf Proof of Corollary~\ref{Cororllary-cubic-3-connected-2-chord}, assuming Theorem~\ref{Theorem-cubic-3-connected}]
For the sake of contradiction, suppose  $G$  contains a longest cycle $C$ that has at most one chord. By Theorem~\ref{Theorem-Thomassen=3-connected}, we have that $C$ has a chord, and hence $C$ has exactly one chord. For convenience, we assume that $C = v_1 v_2 \dots v_s v_1$ and that $e = v_1 v_t$ ($2 < t < s$) is a chord of $C$. However, it follows that $C - v_1 v_s$ is a longest $(v_1, v_s)$-path, and $v_t$ is the unique internal $(C - v_1 v_s)$-bound vertex, a contradiction. This proves Corollary~\ref{Cororllary-cubic-3-connected-2-chord}.
\end{proof}

We organize the remainder of this paper as follows: Section~\ref{Preliminaries} presents some useful lemmas  for subsequent proofs. Sections~\ref{Proof-2-connected} and~\ref{Proof-3-connected} are devoted to the proof of Theorems~\ref{Theorem-cubic-2-connected} and~\ref{Theorem-cubic-3-connected}, respectively.

\section{Preliminaries}\label{Preliminaries}
In this section, we present some notations and useful lemmas which will be used later in our proofs. A $k$-coloring of $G$ is a mapping $\varphi:V(G)\to[k]$ such that $\varphi(u)\neq\varphi(v)$ whenever $u$ and $v$ are adjacent in $G$.
The  chromatic number $\chi(G)$ of $G$
is the minimum integer $k$ such that $G$ admits a $k$-coloring.
Erd\H{o}s \cite{Erdos1990} popularized the intriguing cycle plus triangles problem, which asks whether every $4$-regular graph, formed by combining a Hamilton cycle with vertex-disjoint triangles, can be colored with only three colors. Fleischner and Stiebitz \cite{Fleischner1992} gave a positive solution to this problem. Furthermore, Fleischner and Stiebitz \cite{Fleischner1997} remarked that the conclusion of Theorem~\ref{Lemma-color-triangle} can be inferred from the results obtained in \cite{Fleischner1992}.
\begin{theorem}[Fleischner-Stiebitz  \cite{Fleischner1992}]\label{Lemma-color-triangle}
If a graph $G$ is the edge-disjoint union of a Hamilton cycle and some pairwise vertex-disjoint triangles, then $G$ is $3$-colorable. 
\end{theorem}
Through Theorem~\ref{Lemma-color-triangle}, we establish the following lemma which will be useful in our proofs.
\begin{lemma}\label{Lemma-color-triangle-path}
Let $G$ be a graph with a Hamilton cycle $C$.   If each nontrivial component of $G-E(C)$  is either a triangle or  a path of order $3$, then $G$ is $3$-colorable.
\end{lemma}
\begin{proof}[\bf Proof of Lemma~\ref{Lemma-color-triangle-path}.]
    If all non-trivial components of $G-E(C)$  are triangles, then the result follows from Theorem~\ref{Lemma-color-triangle}. Hence, we may assume $H_1,\ldots ,H_t$ to be  non-trivial components of $G-E(C)$ that are not triangles. Then $H_i$ is a path of order $3$. For convenience, we assume that $H_i=u_iv_iw_i$ for each $i\in [t]$. We construct a new graph $G'$ in the following manner: If $u_iw_i\notin E(C)$, then we add one new edge $u_iw_i$. If $u_iw_i\in E(C)$, then we subdivide the edge $u_iw_i$ in $G$ by one new vertex $z_i$ of degree $2$ and add one new edge $u_iw_i$. Note that the paths
   $u_iz_iw_i$ are used to extend $C$ to a Hamilton cycle in $G'$.
    Now, $G'$ is the edge-disjoint union of a Hamilton cycle and some pairwise vertex-disjoint triangles. By Theorem~\ref{Lemma-color-triangle}, we have that $\chi (G')\leq 3$. Note that $G$ is a  subgraph of $G'$. Hence, $\chi(G)\leq \chi (G')\leq 3$. This proves Lemma~\ref{Lemma-color-triangle-path}. 
\end{proof}
The following theorem was used by Thomassen to prove that every longest cycle in a $2$-connected cubic graph has a chord. This theorem will also be useful in our proofs.
\begin{theorem}[Thomassen \cite{Thomassen2018}]\label{Lemma-match}
Let $G$ be a  cubic graph with $V(G)$ has a partition into sets $A$, $B$ such that the induced graph $G[A]$ is a matching $M$, and $G[B]$ is a matching $M'$. Let $|A|=|B|=2k$. Assume that $G$ has a cycle $C$ of length $3k$ such that $C$ contains each edge in $M$, and precisely one end of each edge in $M'$. Then $G$ has a cycle of length greater than $3k$ containing $M$. 
\end{theorem}
Thomason \cite{Thomason1978} introduced his elegant and powerful so called lollipop method.  We apply the lollipop method and follow the ideas from \cite[Lemma 2.1 and Theorem 2.2]{Thomassen1997} to establish the following results.
\begin{lemma}\label{Lemma-edge-even-number}
 Let $k\geq 2$ be an integer and let $G$ be a graph with a Hamilton cycle $C$. Suppose that for some vertex set $A$ of cardinality $k$, the subgraph $G-A$ has $k$ components $H_1,H_2,\ldots ,H_{k}$ each of which is a path, and the endpoints of $H_i$ are of odd degree in $G$ for $i\in [k-1]$.
Then
\begin{itemize}
    \item[$(1)$]for every Hamilton cycle $C'$ of $G$, $C'-A=C-A$, and
    \item[$(2)$]each edge of $G-E(H_{k})$ incident to an endpoint of $H_{k}$ is included in an even number of Hamilton cycles of $G$.
\end{itemize}
\end{lemma}
\begin{proof}[\bf Proof of Lemma~\ref{Lemma-edge-even-number}]
 Let $O$ be a Hamilton cycle of $G$. Since $O-A$ is a spanning subgraph of $G-A$ and $G-A$ has $k$ components, we have
 \begin{flalign*}
k=c(G-A)\le c(O-A)\le k.
 \end{flalign*}
 Hence, $c(G-A)=c(O-A)$. Combining the fact that $G-A$ is a forest and $O-A$ is a spanning subgraph of $G-A$, we have that $O-A=G-A$. This proves $(1)$.

 Next, we prove $(2)$. Let $v$ be an endpoint of $H_k$, and let $uv$ be an edge of $G - E(H_k)$. If there is no Hamilton cycle containing $uv$, then the statement is trivial. Suppose there exists a Hamilton cycle $O$ containing $uv$. Let $w$ be a neighbor of $u$ in $O$ distinct from $v$. Note that every longest path in $G$ which starts with the edge $uv$ is Hamilton path. 
Let $P$ be a Hamilton path that starts with the edge $uv$ and has $u$ as an endpoint, and let $x$ be the endpoint of $P$ distinct from $u$.
\begin{claim}\label{Claim-endpoint-odd-degree}
$x$ is an endpoint of $H_i$ for some $i \in [k-1]$.
\end{claim}
Suppose $x\in A$. Since $P-A$ is a spanning subgraph of $G-A$, $c(G-A)\leq c(P-A)$.  Note that $u\in A$. However, it follows that 
 \begin{flalign*}
     k= c(G-A)\leq c(P-A)\leq k-1,
 \end{flalign*}
a contradiction. 

Suppose $x$ is an endpoint of $H_k$.  Clearly, $x\neq v$. Since $P-A$ is a spanning subgraph of $G-A$, $k=c(G-A)\leq c(P-A)\leq k$. Hence $G-A=P-A$. Note that $v$ and $x$ belong to the same component of $G-A$. This implies that $G - A = H_k$, and hence $k = 1$, a contradiction.

Suppose $x\in H_i^*$ for some $i\in [k]$. However, it follows that 
\begin{flalign*}
 k+1=c(G-(A\cup \{x\}))\leq c(P-(A\cup \{x\}))\leq k,
\end{flalign*}
a contradiction. Therefore, $x$ is an endpoint of $H_i$ for some $i \in [k-1]$. This proves Claim~\ref{Claim-endpoint-odd-degree}. 

 Now, we consider an auxiliary graph $\mathcal{P}_{uv}$. Its vertex set consists of all Hamilton paths in $G$ that start with the edge $uv$ and have $u$ as an endpoint. Next, we define the adjacency relation on the auxiliary graph $\mathcal{P}_{uv}$. 
 Let $P$ be a vertex of $\mathcal{P}_{uv}$ and let $x$ be the endpoint of $P$ distinct from $u$. Suppose $y$ is a neighbor of $x$ such that $xy\notin E(P)$ and $y\neq u$. Then adding the edge $xy$ to $P$ produces a unique Hamilton path $P'\neq P$ that is a vertex of  $\mathcal{P}_{uv}$. We say that $P$ and $P'$ are adjacent in $\mathcal{P}_{uv}$.

Let $P$ be a vertex of $\mathcal{P}_{uv}$ and let $x$ be the endpoint of $P$ distinct from $u$. By Claim~\ref{Claim-endpoint-odd-degree}, we have that $x$ has odd degree in $G$. Therefore, $P$ has odd degree in $\mathcal{P}_{uv}$ if $ux \in E(G)$. Otherwise, $P$ has even degree in $\mathcal{P}_{uv}$. Note that the number of vertices with odd degree is even in $\mathcal{P}_{uv}$. Therefore, the edge $uv$ is included in an even number of Hamilton cycles of $G$. This completes the proof of Lemma~\ref{Lemma-edge-even-number}. 
 \end{proof}

\begin{lemma}\label{Lemma-another-Hamilton-cycle}
  Let $k\geq 2$ be an integer and let $G$ be a graph with a Hamilton cycle $C$. Suppose that for some vertex set $A$ of cardinality $k$ such that
\begin{itemize}
\item[$(1)$] $C-A$ has $k$ components $H_1,H_2,\ldots,H_k$ (i.e., $A$ is an independent set in $C$), and
\item[$(2)$] every endpoint of $H_i$ is joined to a vertex in $A$ by some chord of $C$ for $i\in [k-1]$.
\end{itemize}
 Then $G$ has a Hamilton cycle $C'$ distinct from $C$. Let $x$ be an endpoint of $H_k$ and let $xy$ be an edge in $C-E(H_k)$. Moreover, $C'$ can be chosen such that 
\begin{itemize}
\item[$(3)$] $C'$ contains the edge $xy$, $C'-A=C-A$, and 
\item[$(4)$] there is a vertex $v$ in $A$ such that one of the two edges of $C'$ incident with $v$ is in $C$ and the other is not in $C$.
\end{itemize}
\end{lemma}
\begin{proof}[\bf Proof of Lemma~\ref{Lemma-another-Hamilton-cycle}]
Without loss of generality, assume that the endpoints of $H_i$ are $x_i$ and $y_i$ for each $i \in [k-1]$. Let $e_{x_i}$ be a chord of $C$ joining $x_i$ to some vertex in $A$ for each $i \in [k-1]$, and let  $e_{y_i}$ be a chord of $C$ joining $y_i$ to some vertex in $A$ for each $i \in [k-1]$. 
Denote  $G_1=G[E(C)\cup (\bigcup\limits_{i=1}^{k-1}\{e_{x_i},e_{y_i}\})]$. Clearly, $G_1 - A = C - A$. Hence, $H_1, \dots, H_k$ are the components of $G_1 - A$. Now, $x_i$ and $y_i$ have degree $3$ in $G_1$ for each $i \in [k-1]$. By Lemma~\ref{Lemma-edge-even-number}, we have that  
\begin{itemize}  
    \item[$(i)$] for every Hamilton cycle $C'$ of $G_1$, $C' - A = C - A$, and  
    \item[$(ii)$] the edge $xy$ is included in an even number of Hamilton cycles of $G_1$.  
\end{itemize}  
Since $C$ is a Hamilton cycle of $G_1$ that contains $xy$, there exists another Hamilton cycle distinct from $C$ of $G_1$ that also contains $xy$.

Next, we prove $(4)$. Let $C_1$ be a Hamilton cycle distinct from $C$ in $G_1$ containing the edge $xy$, and assume that $C_1$ is chosen such that $|E(C_1) \cap E(C)|$ is as large as possible. By $(i)$, we have that $C_1-A=C-A$. 

    Suppose $C_1$ does not satisfy  $(4)$. Therefore, for every vertex $v \in A$, the two edges of $C_1$ incident to $v$ are either both in $C$ or both not in $C$. Denote $G_2=G_1[E(C)\cup E(C_1)]$. This implies that every vertex of $A$ has degree $2$ or $4$ in $G_2$.  Note that $C_1-A=C-A$ and $C_1\neq C$. Hence, there is a vertex in $A$ with degree $4$ in $G_2$. Let $A'$ denote the set of vertices in $A$ that have degree $4$ in $G_2$, and let $|A'|=r$. Clearly, $r\geq2$. 

Next, we analyze the degrees of the vertices in graph $G_2$.  Clearly, $\Delta(G_2)\leq 4$.
Let $u$ be a vertex of $G_2$. If $u\in H_i^*$ for some $i\in [k]$, then $u$ has degree $2$ since $C-A=C_1-A$ and $[A,H_i^*]=\emptyset.$ Suppose that $u$ is an endpoint of $H_i$ for some $i \in [k]$. By the construction of $G_1$, we have $d_{G_2}(u)\le d_{G_1}(u)\le 3$. Furthermore, if $u$ is adjacent to some vertex $v$ of $A'$ in $G_2$, then $u\notin N_C(v)\cap N_{C_1}(v)$ and hence $u$ has degree $3$ in $G_2$; 
if $u$ is not adjacent to any vertex of $A'$ in $G_2$, then $u$ is adjacent to some vertex in $A\setminus A'$ and $u$ has degree $2$. 

Now, $G_2-A'$ has $r$ components $P_1,P_2,\ldots,P_{r}$ each of which is a path, and the endpoints of $P_i$ are of odd degree in $G_2$ for $i\in [r]$. Let $w$ be a vertex of $A'$ and let $ww'$ be an edge of $C_1$. By Lemma~\ref{Lemma-edge-even-number}, we have that 
\begin{itemize}
    \item[$(iii)$]for every Hamilton cycle $C_2$ of $G_2$, $C_2-A'=C_1-A'$, and
    \item[$(iv)$] the edge $ww'$ is included in an even number of Hamilton cycles of $G_2$.
\end{itemize}
By $(iv)$, there exists a Hamilton cycle $C_2$, distinct from $C_1$ in $G_2$ and by $(iii)$ we have $C_2-A'=C_1-A'$.  Since $C_2\neq C_1$, there exists $z\in A'$ such that $C_2$ does not contain a certain edge of $C_1$  incident with $z$. This implies that $C_2$ contains an edge of $C$ incident with $z$.

Let $f \in E(C) \cap E(C_1)$. Since $C_1$ does not satisfy condition $(4)$, $f$ is not incident to any vertex of $A'$. However, by  $(iii)$, it follows that
\begin{flalign*}
f \in E(C_1 - A') = E(C_2 - A') \subseteq E(C_2).
\end{flalign*}
Hence, $E(C_1) \cap E(C)\subseteq E(C_2) \cap E(C)$.
Recall that $C_2$ contains an edge of $C$ incident with $z$, and this edge is not in $C_1$. This implies that
\begin{flalign*}
|E(C_2) \cap E(C)| > |E(C_1) \cap E(C)|,
\end{flalign*}
which contradicts the choice of $C_1$. Therefore, $C_1$ satisfies condition $(4)$. This completes the proof of Lemma~\ref{Lemma-another-Hamilton-cycle}.
\end{proof}

\section{Proof of Theorem~\ref{Theorem-cubic-2-connected}}\label{Proof-2-connected}

The aim of this section is to prove Theorem~\ref{Theorem-cubic-2-connected}. Before proceeding with the proof, we list  some notations and  observations that will be needed in later proofs. 
 Let $G$ be a graph,  let $H$ be a subgraph of $G$ with $H\neq G$, and let $x\in V(G)\setminus V(H)$. We call $x$ a neighbor of $H$ if $N_G(x) \cap V(H) \neq \emptyset$.
\begin{observation}\label{Observation-2-connected}
Let $G$ be a $2$-connected cubic graph, let $C$ be an induced subgraph of $G$, and let $H$ be a component of $G-V(C)$.  If $H$ has exactly two neighbors $x,y$ in $V(C)$, then there exists an $(x,y)$-path of length at least three whose interior is in $V(H)$.
\end{observation}
\begin{proof}[\bf Proof of Observation~\ref{Observation-2-connected}]
Since $G$ is $2$-connected cubic graph, there exist two distinct vertices $x',y'$ such that $x'\in N_H(x)$ and $y'\in N_H(y).$ Let $Q$ be an $(x',y')$-path in $H$. Then $xx'Qy'y$ is an $(x,y)$-path of length at least three whose interior is in $V(H)$. This completes the proof of Observation~\ref{Observation-2-connected}.
\end{proof}
\begin{proof}[\bf Proof of Theorem~\ref{Theorem-cubic-2-connected}]
We prove Theorem~\ref{Theorem-cubic-2-connected} by contradiction. Suppose that there exist two vertices $x,y$  and a longest $(x,y)$-path $P$  such that $P$ contains no internal $P$-bound vertex. 
Since $G$ is a cubic graph, $P$ is an induced path of $G$ if $xy\notin E(G)$, and $P\cup \{xy\}$ is an induced cycle of $G$ if $xy\in E(G)$. 
We may assume that $P$ is not a Hamilton path of $G$. Let $H_1,\ldots, H_t$ be the components of $G-V(P)$. For convenience, assume that $P=xu\ldots vy$. 

\begin{claim}\label{Claim-component}
There exists some $i\in [t]$ such that $N_P(H_i)\cap \{x,y\}=\emptyset$.
\end{claim}
 Suppose for the sake of contradiction that $N_P(H_i)\cap \{x,y\}\neq \emptyset$ for each $i\in [t]$. Since $G$ is a cubic graph, there exists $i\in [t]$ such that $u\in N_P(H_i)$. To avoid an $(x, y)$-path longer than $ P $, we  have that $ x \notin N_P(H_i) $. Then $ y \in N_P(H_i) $, which implies that $ v \notin N_P(H_i) $. By a similar argument, there exists $ j \in [t] $ such that $ x, v \in N_P(H_j) $.    Therefore, there is a $(u,y)$-path $P_{uy}$ whose interior is in $V(H_i)$, and an $(x,v)$-path $P_{xv}$ whose interior  is in $V(H_j)$.
 However, it follows that $ xP_{xv}vP[v,u]uP_{uy}y $ is an $ (x,y) $-path of length greater than the length of $ P $,
a contradiction. This proves Claim~\ref{Claim-component}.

By Claim~\ref{Claim-component}, we may assume that $H_1, \ldots, H_r$ are components of $G-V(P)$ such that  $N_P(H_i) \cap \{x, y\} = \emptyset$ for each $i \in [r]$, where $r<t$. 
Recall that $P^*=V(P[u,v])$. Then $N_P(H_i)\subseteq P^*$ for each $i\in [r]$.  

For each $i\in [r]$, if $|N_P(H_i)|\geq 3$, then we select three neighbors of $H_i$ in $P$, denoted by $u_i,v_i$, and $w_i$.  Since $G$ is a cubic graph, $\{u_i,v_i,w_i\}\cap \{u_j,v_j,w_j\}=\emptyset$ whenever $i\neq j$. In fact, it is possible that $|N_P(H_i)|=2$ for each $i\in [r]$.

Suppose that there exists $i\in [r]$ such that $|N_P(H_i)|\geq 3$. We now define a new graph $G_1$, whose construction is based on  the subpath $P[u,v]$ and $\{u_i,v_i,w_i\}$ for $i \in [r]$, as follows:
 If $\{u,v\}\subseteq \{u_i,v_i,w_i\}$ for some $i\in [r]$, we define
\begin{flalign*}
     G_1=P[u,v]\cup \left( \bigcup_{i=1}^{r} \{u_iv_i,u_iw_i,v_iw_i\}\right).
\end{flalign*}
Otherwise, we define
\begin{flalign*}
G_1=P[u,v]\cup \{uv\}\cup \left( \bigcup_{i=1}^{r} \{u_iv_i,u_iw_i,v_iw_i\}\right).
\end{flalign*}
 Now, each non-trivial component of $G_1-uv-E(P[u,v])$ is either a triangle or a path of order $3$. By Lemma~\ref{Lemma-color-triangle-path}, we have that $\chi(G_1)\leq 3$. Since $G$ contains at least one triangle, $\chi(G_1)= 3$. Therefore, for a $3$-coloring of $G_1$, there exists  a color class $A$ such that $A\cap \{u,v\}=\emptyset$. Under this circumstance, we assume that $A$ contains all vertices of the form $w_i$. 
 Suppose that $|N_P(H_i)|=2$ for each $i\in [r]$. We define $A=\emptyset$. 

Next, we construct a new graph $G_2$ from $G$ as follows:
\begin{itemize}
    \item[$(1)$] For $i\in [t]$, if $H_i$ is joined to only two vertices $x_i,y_i$ on $P$, then we replace $H_i$ with an edge $x_iy_i$.
    \item[$(2)$] For the remaining component $H_i$, where $i\in [r]$,  we contract $H_i$ into  $w_i$.
    \item[$(3)$] For the remaining component $H_i$, if $y$ has a neighbor in $H_i$, then we contract $H_i$ into $y$. Otherwise, we contract it into $x$. 
     \item[(4)] By the previous construct, if $x$ and $y$ are nonadjacent, then we add a new edge $xy$.
\end{itemize}
Let $C=G_2[E(P)\cup \{xy\}]$. It is clear that $C$ is a Hamilton cycle in $G_2$. Next, we color the edges of $G_2$ as follows: an edge in $C$  is colored black, an edge of the form  $x_iy_i$  is colored red, and all remaining edges are colored blue. See Figure~\ref{Figure-construction} for a depiction of the construction of $G_2$.
\begin{figure}[htbp]
    \centering

    \begin{subfigure}[b]{\textwidth}
        \centering
        \scalebox{0.8}{
        \begin{tikzpicture}[thick, every node/.style={circle, draw=black, fill=black, inner sep=2pt}]
            \node[label={[above, yshift=-3mm]:{$x(x_1)$}}] (0) at (-7.5, 5) {};
            \node[label={[above, yshift=-0.5mm]:{$x_2$}}] (1) at (-6, 5) {};
            \node[label=above:{}] (2) at (-3, 5) {};
            \node[label={[above, yshift=-0.5mm]:{$y_1$}}] (3) at (0, 5) {};
            \node[label={[above, yshift=-0.8mm]:{$w_3$}}] (4) at (1.5, 5) {};
            \node[label={[above, yshift=-0.5mm]:{$y_2$}}] (5) at (4.5, 5) {};
            \node[label=above:{}] (6) at (6, 5) {};
            \node[label={[above, yshift=-0.5mm]:{$y$}}] (7) at (7.5, 5) {};
            \node[label=above:{}] (11) at (-4.5, 5) {};
            \node[label=above:{}] (13) at (-1.5, 5) {};
            \node[label=above:{}] (14) at (3, 5) {};

            \node[draw=black, fill=white, inner sep=3pt, label=below:{}] (8) at (-4, 3) {$H_1$};
            \node[draw=black, fill=white, inner sep=3pt, label=below:{}] (9) at (-2, 3) {$H_2$};
            \node[draw=black, fill=white, inner sep=3pt, label=below:{}] (10) at (2, 3) {$H_3$};
            \node[draw=black, fill=white, inner sep=3pt, label=below:{}] (12) at (4, 3) {$H_4$};

            \draw[dotted, line width=1.5pt] 
                (0) -- (1)
                (3) -- (4)
                (5) -- (6)
                (6) -- (7)
                (1) -- (11)
                (11) -- (2)
                (7) -- (12)
                (2) -- (13)
                (13) -- (3)
                (4) -- (14)
                (14) -- (5)
            ;

            \draw[line width=1.5pt]
                (0) -- (8)
                (12) -- (13)
                (8) -- (3)
                (9) -- (1)
                (9) -- (5)
                (10) -- (2)
                (10) -- (4)
                (11) -- (10)
                (10) -- (6)
                (7) -- (12)
                (12) -- (14)
                (12) -- (13)
            ;
        \end{tikzpicture}
        }
    \end{subfigure}

    \vspace{-0.5em}
    \begin{tikzpicture}
        \node at (0,0) {\Huge$\Downarrow$};
    \end{tikzpicture}
    \vspace{-0.5em}

    \begin{subfigure}[b]{\textwidth}
        \centering
        \scalebox{0.8}{
        \begin{tikzpicture}[thick, every node/.style={circle, draw=black, fill=black, inner sep=2pt}]
            \node[label={[below, yshift=-1mm]:{$x(x_1)$}}] (0) at (-7.5, 5) {};
            \node[label={[below, yshift=-4mm]:{$x_2$}}] (1) at (-6, 5) {};
            \node[label=above:{}] (2) at (-3, 5) {};
            \node[label={[below, yshift=-4mm]:{$y_1$}}] (3) at (0, 5) {};
            \node[label={[below, yshift=-4mm]:{$w_3$}}] (4) at (1.5, 5) {};
            \node[label={[below, yshift=-4mm]:{$y_2$}}] (5) at (4.5, 5) {};
            \node[label=above:{}] (6) at (6, 5) {};
            \node[label={[below, yshift=-4mm]:{$y$}}] (7) at (7.5, 5) {};
            \node[label=above:{}] (11) at (-4.5, 5) {};
            \node[label=above:{}] (13) at (-1.5, 5) {};
            \node[label=above:{}] (14) at (3, 5) {};

            \path 
                (0) edge[bend left=30, line width=1.5pt] (7)
                (0) edge[red, bend left=25, line width=1.5pt] (3)
                (1) edge[red, bend left=30, line width=1.5pt] (5)
                (4) edge[blue, bend right=40, line width=1.5pt] (11)
                (4) edge[blue, bend right=25, line width=1.5pt] (2)
                (4) edge[blue, bend left=25, line width=1.5pt] (6)
                (7) edge[blue, bend right=25, line width=1.5pt] (14)
                (7) edge[blue, bend right=25, line width=1.5pt] (13)
            ;

            \draw[dotted, line width=1.5pt] 
                (0) -- (1)
                (3) -- (4)
                (5) -- (6)
                (6) -- (7)
                (1) -- (11)
                (11) -- (2)
                (2) -- (13)
                (13) -- (3)
                (4) -- (14)
                (14) -- (5)
            ;
        \end{tikzpicture}
        }
    \end{subfigure}
\caption{The construction of $G_2$
}
    \label{Figure-construction}
\end{figure}

\begin{claim}\label{Claim-another-cycle}
$G_2$ has a cycle $C'$ distinct from $C$ such that $xy\in E(C')$  and $C'$ contains all vertices  of odd degree in $G_2$.
\end{claim}
As in the lollipop argument,  we denote an auxiliary graph $\mathcal{P}_{xy}$. A vertex in $\mathcal{P}_{xy}$ is a path $P$ in $G_2$ which starts with the edge $xy$,  has $x$ as an endpoint, contains all vertices of odd degree in $G_2$, and ends with a vertex of odd degree in $G_2$.  Next, we define the adjacency relation in the auxiliary graph $\mathcal{P}_{xy}$. Let $P_1$ be a vertex of $\mathcal{P}_{xy}$, and let $x'$ be the endpoint of $P_1$ distinct from $x$. 

Note that each vertex in $V(G_2)\setminus (A\cup \{x,y\})$ has degree $3$. It follows that $P_1$ contains all vertices in $G_2$ except some vertices in $A$. Recall that $A$ is an independent set in $P[u,v]$. By the construction of $G_2$, $A$ is also an independent set in $G_2$.

Consider either an edge $x'w$ or a path $x'zw$,  where $w\in P_1^*$ and $z\in A \setminus V(P_1)$.
We add the edge $x'w$ or the path $x'zw$ to $P_1$. If the vertex succeeding  $w$ on $P_1$ has even degree in $G_2$, then we delete this vertex. Otherwise, we delete only the edge succeeding $w$ on $P_1$. The resulting path $P_2$ is a vertex of $\mathcal{P}_{xy}$.
We say that $P_1$ and $P_2$ are adjacent in $\mathcal{P}_{xy}$.  

Next, we analyze the degree of the vertices in auxiliary graph $\mathcal{P}_{xy}$. For convenience,  assume that $N_{P_1}(x')=\{x_1',\ldots ,x_s'\}$ and $N_{G_2}(x')\setminus N_{P_1}(x')=\{y_1',\ldots ,y_h'\} $.  That is, $d_{G_2}(x')=s+h$. Since $x'$ has degree odd in $G_2$, $s+h$ is odd. 
By the choice of $P_1$, $y_i'$ has degree even in $G_2$ for each $i\in [h]$. Note that $N_{G_2}(\{x,y\})\cap A=\emptyset$
and $\{y_1',\ldots ,y_h'\}\subseteq A$. This implies that $xy_i'\notin E(G_2)$ for each $i\in [h]$. Recall that $A$ is an independent set. Thus, $\{y_1',\ldots ,y_h'\}$ is independent. Therefore, if $x'x\notin E(G_2)$, then 
\begin{flalign*}
d_{\mathcal{P}_{xy}}(P_1)=(s-1)+\sum_{i=1}^h(d_{G_2}(y_i')-1).
\end{flalign*}
Since $s-1$ and $h$  have the same parity,  $d_{\mathcal{P}_{xy}}(P_1)$ is even.

If $x'x\in E(G_2)$, then 
\begin{flalign*}
d_{\mathcal{P}_{xy}}(P_1)=(s-2)+\sum_{i=1}^h(d_{G_2}(y_i')-1).
\end{flalign*}
Since $s-2$ and $h$  have different parities,  $d_{\mathcal{P}_{xy}}(P_1)$ is odd. Note that the number of vertices with odd degree is even in $\mathcal{P}_{xy}$. Since $C - xu$ has odd degree in $\mathcal{P}_{xy}$, there exists another vertex $Q \ne C - xu$ with odd degree in $\mathcal{P}_{xy}$. This proves Claim~\ref{Claim-another-cycle}.

By Claim~\ref{Claim-another-cycle}, we have that $G_2$ has a cycle $C'$ distinct from $C$ such that $xy\in E(C')$  and $C'$ contains all vertices of odd degree in $G_2$. Let $R=V(C)\setminus V(C')$ and let $r=|R|$. Clearly, $R\subseteq A$. Furthermore, let $b$ and $c$ be the number of blue edges and red edges, respectively, in $C'$.

Denote $S=E(C)\setminus E(C')$ and $S'=E(C')\setminus E(C)$. Clearly, $|S'|=b+c.$
Let $|S|=k$. This implies that $|E(C')|=|E(C)|-k+b+c$. Furthermore, denote $d=k-2r$. Note that $R$ is an independent set in $C$. Recall that $R=V(C)\setminus V(C')$ and $|R|=r$. It follows that $C$ has $2r$ edges  that are not in $C'$ and hence $d\geq 0$. One readily observes that $C-S$ has exactly $2r+d$ components.
Note that each vertex in $R$ is an isolated vertex in $C-S$. It follows that $C-R-S$ has exactly $r+d$ components. Therefore, 
$$
r+d=c(C-R-S)=c(C'-S')=b+c.
$$

Let $B_1,B_2,\ldots,B_q$ denote the maximal blue subpaths within $C'$. According to the construction of $G_2$, the subgraph $G_2[A\cup\{x,y\}]$ contains precisely one edge $xy$ which is a black edge. Given that each blue edge is incident to either a vertex in $A$, or to $x$, or to $y$, it follows that the length of every maximal blue path is at most two. Without loss of generality, we assume that the subpaths $B_1,\ldots,B_p$ have a length of two, while the subpaths $B_{p+1},\ldots,B_q$ have a length of one, where $p\leq q$. Consequently, $b=q+p$. 

\begin{claim}\label{Claim-p-d-q}
$d\geq 2p$. Furthermore, if $q\ge p+1$, then $d\geq 2p+1$.
\end{claim}
Without loss of generality,  we assume that $V(B_i) \cap A = \{w_i\}$ for $i \in [p]$. Since each $w_i$ is incident to two blue edges in $C'$ where $i \in [p]$, the two black edges incident to $w_i$ in $C$ must belong to $S$. Since $A$ is an independent set, the two black edges incident to $w_i$ in $C$ are not incident to any vertex in $R$ where $i \in [p]$. Thus, aside from the $2r$ edges in $C$ that are incident to some vertex in $R$, $S$ contains at least $2p$ edges and hence $d \ge 2p$.

Suppose $q\ge p+1$. Then the edge set $E(B_{p+1})$ must be incident to either $x$, $y$, or some vertex in $A\setminus R$. Consequently, in every scenario, there exists a black edge in $C$ that is incident to the relevant vertex (either $x$, $y$, or a vertex in $A\setminus R$) and this black edge belongs to $S$. This implies that $d \ge 2p+1$. This proves Claim~\ref{Claim-p-d-q}.

\begin{figure}[htpb]
    \centering
    \scalebox{0.8}{
    \begin{tikzpicture}[thick]
        \coordinate (A) at (-4,0);
        \coordinate (B) at (-2,0);
         \coordinate (C) at (-1,0);
        \coordinate (D) at (0,0);
        \coordinate (E) at (1,0);
        \coordinate (F) at (2,0);
        \coordinate (G) at (3,0);
        \coordinate (H) at (5,0);

       \coordinate (P1) at (7,0);
        \coordinate (P2) at (9,0);
         \coordinate (P3) at (10,0);
        \coordinate (P4) at (11,0);
        \coordinate (P5) at (12,0);
        \coordinate (P6) at (13,0);
        \coordinate (P7) at (14,0);
        \coordinate (P8) at (16,0);
         \coordinate (P9) at (9.5,-1);
         \coordinate (P10) at (13,-1);
        \path
         (A) edge[bend left=40, line width=1.5pt] (H)
         (A) edge[dotted, line width=1.5pt] (B)
          (B) edge[blue, line width=1.5pt] (C)
          (C) edge[ line width=1.5pt] (D)
            (D) edge[dotted, line width=1.5pt] (E)
            (E) edge[blue, line width=1.5pt] (F)
            (F) edge[blue, line width=1.5pt] (G)
             (G) edge[dotted, line width=1.5pt] (H)
             
            (P1) edge[bend left=40, line width=1.5pt] (P8)
            (P1) edge[dotted, line width=1.5pt] (P2)
           
            (P3) edge[ line width=1.5pt] (P4)
            (P4) edge[dotted, line width=1.5pt] (P5)
            
             (P7) edge[dotted, line width=1.5pt] (P8)
             (P9) edge[ line width=1.5pt] (P3)
             (P10) edge[ line width=1.5pt] (P5)
             (P10) edge[ line width=1.5pt] (P7)
             (P9) edge[ line width=1.5pt] (P2);

        \tikzset{std node fill/.style={circle, fill=black, draw=black, line width=1pt, inner sep=2pt}}
        \foreach \point in {A,B,C,D,E,F,G,H,P1,P2,P3,P4,P5,P6,P7,P8} {
            \node[std node fill] at (\point) {};
        }

         \node[draw=black, circle,fill=white, inner sep=3pt, label=below:{}] (P9) at (9.5, -1) {$H_i$};
         
        \node[draw=black, circle,fill=white, inner sep=3pt, label=below:{}] (P10) at (13, -1) {$H_j$};
        \node at (-1, 0.4) {$w_i$};
       
        \node at (10, 0.4) {$w_i$};
        \node at (13, 0.4) {$w_j$};
        \node at (2, 0.4) {$w_j$};

    \end{tikzpicture}
    }
    \caption{Illustration of paths corresponding to 
the maximal blue subpaths}
    \label{Figure-cycle-blue}
\end{figure}

One readily observes that each $B_i$ in $C'$ corresponds to a path of length at least two with the same endpoints as $B_i$ in $G$, whose interior is in $V(G)\setminus V(P)$, for $i\in [q]$. See the illustration in Figure~\ref{Figure-cycle-blue}.
By Observation~\ref{Observation-2-connected}, a red edge $x_iy_i$ in $C'$ corresponds to an $(x_i,y_i)$-path of length at least three in $G$, whose interior is in $V(G)\setminus V(P)$.  Therefore, $C'$ can be transformed into a cycle $C^*$ in $G$ by 
substituting all $B_i$ for $i\in [q]$ and all red edges within $C'$ with the corresponding paths in $G$ that share the same endpoints. Note that $xy\in E(C^*)$. 
\begin{claim}\label{Claim-c=p=0}
    $r=c=0$ and $q=p$.
\end{claim}
Since $|E(C')|=|E(C)|-k+b+c$, by the construction of $C^*$, we have that
\begin{flalign*}
 |E(C^*)|\geq |E(C)|-k+2q+3c.   
\end{flalign*}
Recall that $k=2r+d$, $b=q+p$, and $b+c=r+d$.
Thus, since $C$ is a longest cycle in $G$, we have that
\begin{align*}
    |E(C)|\geq |E(C^*)|&\geq |E(C)|-2r-d+2q+3c\\
    &=|E(C)|-2r-d+2q+3c+2(r+d-b-c)\\
    &=|E(C)|+c+d+2q-2b\\
    &=|E(C)|+c+d-2p\\
    &\ge |E(C)|,
\end{align*}
where the last inequality follows from Claim~\ref{Claim-p-d-q}. Moreover, the above inequalities must be equalities. Consequently, we obtain $c+d-2p=0$, which implies that $c=0$ and $q=p$ by applying Claim~\ref{Claim-p-d-q} again. It also follows that $2q+3c-2r-d=0$, which implies that $r=0$. This proves Claim~\ref{Claim-c=p=0}.

By Claim~\ref{Claim-c=p=0}, we conclude that $(i)$ the cycle $C'$ has the same vertex set as $C$, $(ii)$ the cycle $C'$ consists only of black and blue edges, and $(iii)$ each maximal blue subpath has a length of exactly two.
From the proof of Claim~\ref{Claim-c=p=0}, it is clear that every maximal blue subpath $B_i$ in $C'$ corresponds one-to-one to a path $Q_i$ of length precisely two. Otherwise, we would obtain $|E(C^*)|>|E(C)|$, a contradiction. Additionally, we denote the unique internal vertex of $Q_i$ as $w_i^*$. See the illustration in Figure~\ref{Figure-cycle}.
\begin{figure}[htpb]
    \centering
    \scalebox{0.8}{
    \begin{tikzpicture}[thick]
        \coordinate (A) at (-4,0);
        \coordinate (B) at (-1,0);
        \coordinate (C) at (0,0);
        \coordinate (D) at (1,0);
        \coordinate (E) at (4,0);

        \coordinate (P1) at (7,0);
        \coordinate (P2) at (10,0);
        \coordinate (P3) at (11,0);
        \coordinate (P4) at (12,0);
        \coordinate (P5) at (15,0);
        \coordinate (P6) at (11,-1);

        \path
            (C) edge[blue, line width=1.5pt] (D)
            (D) edge[blue, line width=1.5pt] (B)
            (A) edge[dotted, line width=1.5pt] (B)
            (D) edge[dotted, line width=1.5pt] (E)
            (A) edge[bend left=40, line width=1.5pt] (E)
            (P1) edge[dotted, line width=1.5pt] (P2)
            (P4) edge[dotted, line width=1.5pt] (P5)
            (P1) edge[bend left=40, line width=1.5pt] (P5)
            (P6) edge[ line width=1.5pt] (P2)
            (P6) edge[ line width=1.5pt] (P4);

        \draw[opacity=0.4,red, line width=4pt] (P1) to [bend left=40] (P5);
        \draw[opacity=0.4,red, line width=4pt] (P6) to (P4);
        \draw[opacity=0.4,red, line width=4pt] (P6) to (P2);
        \draw[opacity=0.4,red, line width=4pt] (P1) to (P2);
        \draw[opacity=0.4,red, line width=4pt] (P5) to (P4);

        \tikzset{std node fill/.style={circle, fill=black, draw=black, line width=1pt, inner sep=2pt}}
        \foreach \point in {A,B,C,D,E,P1,P2,P3,P4,P5,P6} {
            \node[std node fill] at (\point) {};
        }

        \node at (-4, 0.4) {$x$};
        \node at (4, 0.4) {$y$};
        \node at (0, 0.4) {$w_i$};
        \node at (7, 0.4) {$x$};
        \node at (15, 0.4) {$y$};
        \node at (11, 0.4) {$w_i$};
        \node at (11, -1.4) {$w_i^*$};
    \end{tikzpicture}
    }
    \caption{Illustrations of $C'$ and $C^*$}
    \label{Figure-cycle}
\end{figure}

Now, let $G_3$ denote the graph consisting of $C^*$, $C$, and all edges of the form $w_iw_i^*$ 
for each $ i \in [p] $.  See the illustration in Figure~\ref{Figure-G_3}.
\begin{figure}[htbp]
    \centering
    \scalebox{0.8}{
    \begin{tikzpicture}[thick, every node/.style={circle, draw=black, fill=black, inner sep=2pt}]
        \node[label=above:{$x$}] (0) at (-7.5, 5) {};
        \node[label=above:{}] (1) at (-5.5, 5) {};
        \node[label=above:{}] (2) at (-3.5, 5) {};
        \node[label=below right:{$w_2$}] (3) at (0, 5) {};
        \node[label=above:{}] (4) at (1, 5) {};
        \node[label=below right:{$w_3$}] (5) at (4.5, 5) {};
        \node[label=above:{}] (6) at (5.5, 5) {};
        \node[label=above:{$y$}] (7) at (7.5, 5) {};
        \node[label= below right:{$w_1$}] (11) at (-4.5, 5) {};
        \node[label=above:{}] (13) at (-1, 5) {};
        \node[label=above:{}] (14) at (3.5, 5) {};

       \node[label=below:{$w_1^*$}] (15) at (-4.5, 3.5) {};
         \node[label=below:{$w_2^*$}] (16) at (0, 3.5) {};
          \node[label=below:{$w_3^*$}] (17) at (4.5, 3.5) {};
        \path 
(0) edge[bend left=30, line width=1.5pt] (7)
 (1) edge[ bend left=25, line width=1.5pt] (3)
(6) edge[ bend right=25, line width=1.5pt] (3)
(5) edge[ bend right=25, line width=1.5pt] (2)
(5) edge[ bend right=25, line width=1.5pt] (4)
(11) edge[ bend left=30, line width=1.5pt] (13)
(11) edge[ bend left=28, line width=1.5pt] (14)

 ;
        \draw[dotted, line width=1.5pt] 
            (0) -- (1)

            (6) -- (7)

            (2) -- (13)
            
            (4) -- (14)
           
           ;

         \draw[ line width=1.5pt] 
            (15)--(1)
            (15)--(2)
            (15)--(11)
            
            (16)--(13)
            (16)--(3)
            (16)--(4)
            
            (17)--(14)
            (17)--(5)
            (17)--(6)

            ;
    \end{tikzpicture}
    }
    \caption{Illustration of $G_3$}
    \label{Figure-G_3}
\end{figure}

Note that $\{w_iw_i^*:~i\in[q]\}$ forms a matching $M'$.
Let $ G_4 $ be the graph obtained from $ G_3 $ by replacing each maximal path of $C^*$, whose internal vertices have degree $2$ in $ G_3 $, with a single edge between its endpoints; these new edges form a matching $ M $ in $ G_4 $. 
We remark that each maximal path of $C^*$, whose internal vertices have degree $2$ in $G_3$, is also contained in $C$. In addition, $xy$ lies in one of the maximal paths.
The cycle obtained from $ C^* $ after this operation is denoted by $ C_1 $.

We now apply Theorem~\ref{Lemma-match} to $G_4$. By Theorem~\ref{Lemma-match}, the graph $ G_4 $ contains a cycle $ C_1' $ that includes all edges in $ M $ and is longer than $ C_1 $. Note that each edge $w_iw_i^*$ in $G_3$ corresponds to a $(w_i,w_i^*)$-path in $G$ with interior in $V(G)\setminus V(P)$. We now restore the edges of $M$ on $ C_1 $ to their original maximal paths, which were previously compressed from paths where all internal vertices had degree $2$ in $ G_3 $, and replace each edge $w_iw_i^*$ of $M'$ on $C_1$ with its corresponding $(w_i,w_i^*)$-path in $G$. The cycle obtained after these operations is denoted by $ C_1'$. Note that $ C_1' $ contains the edge $ xy $, and the length of $ C_1' $ is greater than the length of $C$. However, it follows that $ C_1' - xy $ is an $(x, y)$-path with length greater than the length of $ P $,  a contradiction. 
This completes the proof of Theorem~\ref{Theorem-cubic-2-connected}.
\end{proof}

\section{Proof of Theorem~\ref{Theorem-cubic-3-connected}}\label{Proof-3-connected}
We prove Theorem~\ref{Theorem-cubic-3-connected} by contradiction. Let $G$ be a counterexample to Theorem~\ref{Theorem-cubic-3-connected}. Then, there exists an edge $xy$ and a longest $(x, y)$-path $P$ in $G$ such that $P$ contains at most one internal $P$-bound vertex. Let $C = G[E(P) \cup \{xy\}]$.  By Theorem~\ref{Theorem-cubic-2-connected}, $P$ contains an internal $P$-bound vertex and so $C$ has a chord.
Since $P$ contains at most one internal $P$-bound vertex, $C$ has exactly one chord, say $e$, and one of the endpoints of $e$ must belong to $\{x, y\}$.
Combining the fact that $G$ is a $3$-connected graph, it is easy to see that $C$ has length at least $5$.
 For convenience, we may assume that $e = xw$. Let $H_1, \ldots, H_t$ be the components of $G - C$.
\begin{claim}\label{Claim-3-connected-component}
     $t\geq 2$.
 \end{claim}
Suppose for the sake of contradiction that $t\leq1$. Clearly, $C$ is not a Hamilton cycle and so $t = 1$. Since $C$ has length at least $5$, $H_1$ has two consecutive neighbors in $P$. This implies that $G$ contains  an $(x,y)$-path of length greater than  the length of $P$, a contradiction. This proves Claim~\ref{Claim-3-connected-component}.

Note that $G$ is $3$-connected. Therefore, for each $i \in [t]$, we can select three neighbors of $H_i$ in $C$, denoted by $u_i$, $v_i$, and $w_i$. Since $G$ is a cubic graph, $\{u_i,v_i,w_i\}\cap \{u_j,v_j,w_j\}=\emptyset$ whenever $i\neq j$. For convenience,  suppose that $N_C(x)=\{a,y,w\}$ and $N_C(w)=\{x,b,c\}$.  Let $G_1$ be a new graph obtained from $G-\{x,w\}$ by adding two edges $ay$ and $bc$, and let $C'$ be a cycle obtained from 
 $C-\{x,w\}$ by adding two edges $ay$ and $bc$. Since $C$ has length at least $5$, we have either $a\neq b$ or $c\neq y$ and so $C'$ is a cycle in $G_1$. See the illustrations of $C$ and $C'$ in Figure~\ref{Figure-cycle-3-connected}.

\begin{figure}[htpb]
    \centering
    \scalebox{0.8}{
    \begin{tikzpicture}[thick]
        \coordinate (A) at (-4,0);
        \coordinate (B) at (-1,0);
        \coordinate (C) at (0,0);
        \coordinate (D) at (1,0);
        \coordinate (E) at (4,0);
        \coordinate (F) at (-3,0);

        \coordinate (P1) at (7,0);
        \coordinate (P2) at (10,0);
        
        \coordinate (P4) at (12,0);
        \coordinate (P5) at (15,0);

        \path
            (C) edge[ line width=1.5pt] (D)
            (D) edge[line width=1.5pt] (B)
            (F) edge[dotted, line width=1.5pt] (B)
            (A) edge[ line width=1.5pt] (F)
            (D) edge[dotted, line width=1.5pt] (E)
            (A) edge[bend left=40, line width=1.5pt] (E)
             (A) edge[bend left=30, line width=1.5pt] (C)
            (P1) edge[dotted, line width=1.5pt] (P2)
            (P4) edge[dotted, line width=1.5pt] (P5)
            (P1) edge[bend left=40, line width=1.5pt] (P5)
            (P2) edge[ line width=1.5pt] (P4)
            ;


        \tikzset{std node fill/.style={circle, fill=black, draw=black, line width=1pt, inner sep=2pt}}
        \foreach \point in {A,B,C,D,E,F,P1,P2,P4,P5} {
            \node[std node fill] at (\point) {};
        }

        \node at (-4, -0.4) {$x$};
        \node at (-3, -0.4) {$a$};
         \node at (-1, -0.4) {$b$};
          \node at (1, -0.4) {$c$};
        \node at (4, -0.4) {$y$};
        \node at (0, -0.4) {$w$};
        \node at (7, -0.4) {$a$};
        \node at (15, -0.4) {$y$};
        \node at (10, -0.4) {$b$};
        \node at (12, -0.4) {$c$};
        
    \end{tikzpicture}
    }
    \caption{Illustrations of $C$ and $C'$}
    \label{Figure-cycle-3-connected}
\end{figure}

We finish the proof of Theorem~\ref{Theorem-cubic-3-connected} by considering the following two cases.

\begin{case}\label{Case1}
$\{a,y\}\cap \{b,c\}=\emptyset.$
\end{case}
We now define a new graph $G_2$, whose construction is based on  the subpath $C'$ and $\{u_i,v_i,w_i\}$ for $i \in [t]$, as follows:
\begin{flalign*}
G_2=C'\cup \bigg(\bigcup\limits_{i=1}^t\{u_{i}v_{i}, u_{i}w_{i}, v_{i}w_{i}\}\bigg).
\end{flalign*}
Note that if there exists $i \in [t]$ such that $\{a,y\}\subseteq \{u_i,v_i,w_i\}$ or $\{b,c\}\subseteq \{u_i,v_i,w_i\}$, then $G_2$ may have parallel edges.
One readily observes that the parallel edge in $G_2$ is either parallel to $ay$ or to $bc$. Furthermore, let $G_2^*$ be a graph obtained from $G_2$ by deleting all possible parallel edges. Then each non-trivial component of $G_2^*-E(C')$ is either a triangle or a path of order $3$.  By Lemma~\ref{Lemma-color-triangle-path}, we have that $\chi(G_2^*)\leq 3$. Since $G_2^*$ contains at least one triangle, we have that $\chi(G_2^*)=3$. 
 Let $A$ be a color
class such that $\{a,y\}\cap A =\emptyset$. Note that $|A\cap \{u_i,v_i,w_i\}|=1$ for each $i\in [t]$. Without loss of generality, we may  assume that $A$ contains all vertices of the form $w_{i}$. By Claim~\ref{Claim-3-connected-component},  we have that $|A|\geq 2$.
Furthermore, we construct a new graph $G_3$ from $G_1$ by contracting each component $H_i$ of $G_1-V(C')$ into $w_{i}$.

Now, $G_3$ is a hamiltonian graph with a Hamilton cycle $C'$ and satisfies the following conditions:
\begin{itemize}
    \item[$(1)$] $C'-A$ has $|A|$ components, and 
    \item[$(2)$] the endpoint of each component in $C' - A$ that does not contain $b$ and $c$ is joined to a vertex in $A$ by some chord of $C'$.
\end{itemize}

By Lemma~\ref{Lemma-another-Hamilton-cycle}, $G_3$ has a Hamilton cycle $C_1$ distinct from $C'$ such that 
\begin{itemize}
    \item[(3)] $\{ay,bc\}\subseteq E(C_1)$, and
    \item[(4)]  there is a vertex $v$ in $A$ such that one of the two edges of $C_1$ incident with $v$ is in $C'$ and the other is not in $C'$.
\end{itemize}
We emphasize that $ay \in E(C_1)$ follows from $\{a, y\} \cap A = \emptyset$. Irrespective of whether $\{b, c\} \cap A = \emptyset$, Lemma~\ref{Lemma-another-Hamilton-cycle} guarantees the existence of a Hamilton cycle $C_1$ containing $bc$.

Now we color the edges of the graph $G_3$ to facilitate our subsequent proof. In particular, we color the edges of the cycle 
$C'$ black and the remaining edges blue.

Let $B_1,B_2,\dots, B_q$ denote the maximal blue subpaths within $C_1$. Since $A$ is an independent set, it follows from the construction of $G_3$ that the length of each $B_i$ is at most two. Without loss of generality, we assume that the subpaths $B_1,\dots, B_p$ have a length of two, while the subpaths $B_{p+1},\dots,B_q$ have a length of one, where $p\le q$.

One readily observes that each $B_i$ in $C_1$ corresponds to a path of length at least two with the same endpoints as $B_i$ in $G_1$, whose interior is in $V(G_1)\setminus V(C_1)$, for $i\in [q]$. See the illustration in Figure~\ref{Figure-cycle-blue}.
Therefore, $C_1$ can be transformed to a cycle $C_1'$ in $G_1$ by substituting all $B_i$ for $i\in [q]$ with the corresponding paths in $G_1$ that share the same endpoints. 
Since $ay$ and $bc$ are colored black, the cycle $C_1'$ contains $ay$ and $bc$. Furthermore, the cycle $C_1'$ can be transformed into a cycle $C_1^*$ in $G$ by deleting edges $ay$ and $bc$, and adding edges $ax,xy,bw$ and $wc$. Thus, 
\begin{align*}
|E(C_1^*)|\ge |E(C_1')|+2
&\ge |E(C_1)|-\sum_{i=1}^p|E(B_i)|-\sum_{i=p+1}^q|E(B_i)|+2q+2\\
&=|E(C)|-\sum_{i=1}^p|E(B_i)|-\sum_{i=p+1}^q|E(B_i)|+2q\\
&=|E(C)|-2p-(q-p)+2q\\
&=|E(C)|+q-p.
\end{align*}
The condition $(4)$ implies that $q>p$ and so $|E(C_1^*)|>|E(C)|.$
However, it follows that $C_1^*-xy$ is an $(x,y)$-path  with length greater than the length of $P$, a contradiction.

\begin{case}
$\{a,y\}\cap \{b,c\}\neq \emptyset.$
\end{case}
Without loss of generality, we may assume that $y=c$. Suppose that there exists a component $H$ of $G_1-V(C')$ such that $\{a,y\}\subseteq N_{C'}(H)$. Let $P_{ay}$ be a $(a,y)$-path with interior in $H$. However, it follows that 
$xwbP[b,a]aP_{ay}y$
is an $(x,y)$-path with length greater than the path of $P$, a contradiction. Therefore,  $|\{a,y\}\cap \{u_i,v_i,w_i\}|\le 1$ for each $i\in [t]$. 

We now define a new graph $G_4$, whose construction is based on   $C'$ and $\{u_i,v_i,w_i\}$ for $i \in [t]$, as follows:
\begin{flalign*}
 G_4=C'\cup \bigg(\bigcup\limits_{i=1}^t\{u_{i}v_{i}, u_{i}w_{i}, v_{i}w_{i}\}\bigg).
\end{flalign*}
Note that if there exists $i \in [t]$ such that $\{b,c\}\subseteq \{u_i,v_i,w_i\}$, then $G_4$ may have parallel edges.
One readily observes that the only possible parallel edge is 
$bc$. Furthermore, let $G_4^*$ be a graph obtained from $G_4$ by deleting the possible parallel edge. Then each non-trivial component of $G_4^*-E(C')$ is either a triangle or a path of order $3$. By Lemma~\ref{Lemma-color-triangle-path}, we have that $\chi(G_4^*)\leq 3$. Since $G_4^*$ contains at least one triangle, we have that $\chi(G_4^*)=3$. 
 Let $A$ be a color
class such that $\{a,y\}\cap A =\emptyset$. Note that $|A\cap \{u_i,v_i,w_i\}|=1$ for each $i\in [t]$. Without loss of generality, we may  assume that $A$ contains all vertices of the form $w_{i}$.
Furthermore, we construct a new graph $G_5$ by contracting each component $H_i$ of $G_1-V(C')$ into $w_{i}$.

Now, $G_5$ is a hamiltonian graph with a Hamilton cycle $C'$ and satisfies the following conditions:
\begin{itemize}
    \item[$(1^*)$] $C'-A$ has $|A|$ components, and 
    \item[$(2^*)$] the endpoint of each component in $C' - A$ that does not contain $b$  is joined to a vertex in $A$ by some chord of $C'$.
\end{itemize}
By  Lemma~\ref{Lemma-another-Hamilton-cycle}, $G_5$ has a Hamilton cycle $C_2$ distinct from $C'$ such that 
\begin{itemize}
    \item[$(3^*)$] $\{ay,yb\}\subseteq E(C_2)$, and
    \item[$(4^*)$]  there is a vertex $v$ in $A$ such that one of the two edges of $C_2$ incident with $v$ is in $C'$ and other is not in $C'$.
\end{itemize}

Similar to Case~\ref{Case1}, $C_2$ can be transformed into a cycle $C_2'$ in $G_1$ by adding paths in certain components of $G' - V(C')$; each such path corresponds to one or two contracted edges incident with a vertex of $A$. Furthermore, the cycle $C_2'$ can be transformed into a cycle $C_2^*$ in $G$ by deleting edges $ay$ and $yb$, and adding edges $ax,xy,yw$ and $wb$. 
 The condition $(4^*)$ implies that the length of $C_2^*$ is greater than the length of $C$. However, it follows that $C_2^*-xy$ is an $(x,y)$-path  with length greater than the length of $P$, a contradiction.
 This completes the proof of Theorem~\ref{Theorem-cubic-3-connected}.\hfill $\Box$

\section*{Acknowledgement}  The authors are grateful to Professor Xingzhi Zhan for his constant support and guidance. This research  was supported by the NSFC grant 12271170 and Science and Technology Commission of Shanghai Municipality (STCSM) grant 22DZ2229014.

\section*{Declaration}

	\noindent$\textbf{Conflict~of~interest}$
	The authors declare that they have no known competing financial interests or personal relationships that could have appeared to influence the work reported in this paper.
	
	\noindent$\textbf{Data~availability}$
	Data sharing not applicable to this paper as no datasets were generated or analysed during the current study.


\begin{thebibliography}{99}

 \bibitem{Alspach1985}
 B.~R. Alspach and C.~D. Godsil (Eds.), 
 \newblock Cycles in Graphs, 
 \newblock\emph{Ann. Discrete Math.}, \textbf{27} (1985) 461--468.
 
 \bibitem{Birmele2008}
 E.~Birmel\'{e}, 
 \newblock Every longest circuit of a 3-connected, $K_{3,3}$-minor free graph has a chord, 
  \newblock\emph{J. Graph Theory}, \textbf{58} (2008) 293--298.


 \bibitem{Erdos1990}
 P.~Erd\H{o}s, 
\newblock On some of my favourite problems in graph theory and block designs, 
\newblock \emph{Matematiche}, \textbf{45}  
(1990) 61--74.

\bibitem{Fleischner1992}
 H.~Fleischner and M.~Stiebitz,
 \newblock A solution to a colouring problem of P. Erd\H{o}s, \newblock \emph{Discrete Math.}, \textbf{101} (1992) 39--48.
 
 \bibitem{Fleischner1997} 
 H.~Fleischner and M.~Stiebitz,
\newblock Some remarks on the cycle plus triangles problem, 
\newblock in: The Mathematics of Paul Erd\H{o}s II, Springer, Berlin, 1997, pp. 136--142.


 
\bibitem{Kawarabayashi2007}
K.~Kawarabayashi, J.~Niu and C.~Q.~Zhang, 
\newblock Chords of longest circuits in locally planar graphs, \newblock \emph{European J. Combin.}, \textbf{28} (2007) 315–321.

\bibitem{Li2003-1}
X.~Li and C.~Q.~Zhang,
\newblock Chords of longest circuits in $3$-connected graphs, \newblock \emph{Discrete Math.}, \textbf{268}  (2003) 199--206.

\bibitem{Li2003-2}
X.~Li and C.~Q.~Zhang, 
\newblock Chords of longest circuits of graphs embedded in torus and Klein bottle, 
\newblock \emph{J. Graph Theory}, \textbf{43} (2003) 1--23.

\bibitem{Thomason1978}
A.~G.~Thomason, 
\newblock Hamilton cycles and uniquely edge colourable graphs,
\newblock \emph{Ann. Discrete Math.}, \textbf{3} (1978) 259--268.

\bibitem{Thomassen1989}
C.~Thomassen, 
\newblock Configurations in graphs of large minimum degree, connectivity, or chromatic number, in Combinatorial Mathematics: Proceedings of the Third International Conference, New York, 1985;
Ann. New York Acad. Sci., Vol.555, 1989, pp. 402-412.


\bibitem{Thomassen1997}
C.~Thomassen, 
\newblock Chords of longest cycles in cubic graphs, 
\newblock \emph{J. Combin. Theory Ser. B}, \textbf{71} (1997) 211--214.



\bibitem{Thomassen2018}
C.~Thomassen, 
\newblock Chords in longest cycles,
\newblock \emph{J. Combin. Theory Ser. B}, \textbf{129} (2018) 148--157.

\bibitem{Wu2014}
J.~Wu, H.~Broersma and H.~Kang, 
\newblock Removable edges and chords of longest cycles in $3$-connected graphs, 
\newblock \emph{Graphs Combin.}, \textbf{30}  (2014) 743--753.

\bibitem{Zhan2024}
X.~Zhan, 
\newblock A conjecture generalizing {T}homassen's chord conjecture in graph theory, 
\newblock \emph{Bull. Iranian Math. Soc.}, \textbf{50} (2024) Paper No. 69, 4pp.

\bibitem{Zhang1987}
C.~Q.~Zhang,
\newblock Longest cycles and their chords,
\newblock \emph{J. Graph Theory}, \textbf{11} (1987) 521--529.





	

  
\end{thebibliography}
\end{document}